\documentclass[10pt,twoside,final]{amsart}

\usepackage[english]{babel}
\usepackage{graphicx,epstopdf,epsfig}
\usepackage{amsfonts,epsfig,fancyhdr,graphics,amsmath,amssymb}

\title{Products of nonprimary cyclic conjugacy classes in the general linear group}

\newtheorem{definition}{Definition}[section]
\newtheorem{theorem}{Theorem}[section]

\newtheorem{lemma}[theorem]{Lemma}
\newtheorem{corollary}[theorem]{Corollary}
\newtheorem{remark}[theorem]{Remark}

\newcommand {\ch }{\operatorname{char}}
\newcommand {\SUG }{\operatorname{SU}}
\newcommand {\UG }{\operatorname{U}}
\newcommand {\PSL }{\mathrm{PSL}}
\newcommand {\PSU }{\mathrm{PSU}}
\newcommand {\PSp }{\mathrm{PSp}}
\newcommand {\mr }{\mathrm{mr}}
\newcommand {\PC }{\mathrm{PC}}
\newcommand {\SC }{\mathrm{SC}}
\newcommand {\Jordan }{\mathrm{J}}
\newcommand {\GL }{\mathrm{GL}}
\newcommand {\SL }{\mathrm{SL}}
\newcommand {\M }{\mathrm{M}}
\newcommand {\Idm }{\mathrm{I}}
\newcommand {\rank }{\operatorname{rank}}
\newcommand {\GF }{\mathrm{GF}}

\usepackage{ifdraft}
\ifdraft{\usepackage[outer]{showlabels}
	\usepackage{todonotes}}{}
\ifdraft{\usepackage[outer]{showlabels}
	\usepackage{todonotes} \usepackage{listlbls} \usepackage{datetime}}{}

\usepackage[pagebackref,colorlinks,citecolor=red,urlcolor=blue,linkcolor=green, bookmarks=false,hypertexnames=true]{hyperref}
\usepackage{orcidlink}
\usepackage{fancyhdr}
\usepackage{verbatim}

\begin{document}
\bibliographystyle{plain}

\setcounter{page}{1}

\thispagestyle{empty}

\keywords{ conjugacy classes, matrix decomposition, cyclic matrices, Schur complement, Thompson conjecture, unitary group}

\subjclass{15A23}

\author{Klaus Nielsen}\,\orcidlink{0009-0002-7676-2944}
\email{klaus@nielsen-kiel.de}

\ifdraft{\today \ \currenttime}{}
\pagestyle{fancy}
\fancyhf{}
\fancyhead[OC]{Klaus Nielsen}
\fancyhead[EC]{Products of nonprimary cyclic conjugacy classes }
\fancyhead[OR]{\thepage}
\fancyhead[EL]{\thepage}
\ifdraft{\title{Products of nonprimary cyclic conjugacy classes in the general linear group v2}}
{}
\maketitle

\begin{abstract}
A cyclic square matrix (and its conjugacy class) $C$ over a field $K$ is called $(m,k)$-cyclic if it has a decomposition $C = A \oplus B$, where $\dim A = m, \dim B = k$ and $m, k \ne 0$. It is shown that the product of
two nonsingular $(m,k)$-cyclic conjugacy classes $\Omega$ and $\Psi$ of $\GL(m+k,K)$ contains all nonscalar matrices $P \in \GL(m+k,K)$ with determinant $\det P = \det \Omega \Psi$.
\end{abstract}

\section{Introduction} \label{intro-sec}

We are interested in the following problem. Let $\Omega$ and $\Psi$ be cyclic conjugacy classes of the general linear group $\GL(n,K)$ over a field $K$. Let
$M \in \GL(n,K)$ be a nonscalar matrix with $\det M = \det \Omega \Psi$. Under what conditions of $\Omega, \Psi$ and $M$ does the product
$\Omega \Psi$ contain $M$?
It is known that  $\Omega \Omega^{-1}$ contains no transvection if $\Omega$ is irreducible.

It follows from a theorem of Sourour \cite[Theorem 1]{Sourour-1986} that if $\Omega$ and $\Psi$ are diagonalizable, then their product  contains all nonscalar matrices $M$ with $\det M = \det \Omega \Psi$. 

If  $|K| \ge 4$, $n \ge 3$, and $\Omega$ and $\Psi$ are triangularizable, then by a theorem of Botha \cite[Main Theorem]{Botha-2010},
$\Omega \Psi$ contains all matrices $M$ with determinant $\det M = \det \Omega \Psi$ and minimal rank at least 2.

The following result is due to A. Lev:
\begin{theorem}{\rm \cite[Theorem 2]{Lev-1994}} \label{Lev}
	Let $|K| \ge 4$, $n \ge 3$.  Let $\Omega$ and $\Psi$ be cyclic similarity classes of $\GL(n,K)$. Assume that $\Omega$ or $\Psi$ is triangularizable.
	 Then $\Omega \Psi$ contains all nonscalar  matrices $M \in \GL(n,K)$ with determinant $\det M = \det \Omega \Psi$.
\end{theorem}

F. Bünger and the author have shown the following results;

\begin{theorem}{\rm \cite[Theorem 1.2]{BN-1999}} \label{BN}
	Let $n \ge 2 \footnote{\cite[Theorem 1.2]{BN-1999} assumes $n \ge 3$, this is not necessary in our version}$.
	Let $\Omega$ and $\Psi$ be cyclic conjugacy classes of $\GL(n,K)$. Assume that one of them is $(1,n-1)$-cyclic. 
	Then $\Omega \Psi$ contains all non-scalar matrices $M \in \GL(n,K)$ with determinant $\det M = \det \Omega \Psi$.
\end{theorem}

\begin{theorem}{\rm \cite[Theorem 4.2]{BN-1999}} \label{BN-2}
	Let $n \ge 2$. Let $\Omega$ and $\Psi$ be cyclic similarity classes of $\GL(n,K)$. Then $\Omega \Psi$ contains all nonprimary matrices $M \in \GL(n,K)$ with determinant $\det M = \det \Omega \Psi$.
\end{theorem}

Theorem \ref{BN-2} is a generalization of a result of Lev 
\cite[Lemma 5]{Lev-1993}, who additionally assumes that $M$ has an eigenvalue in $K$.
We give a new short proof of \ref{BN}.
 For a short proof of \ref{BN-2} see \cite[Lemma 1.5]{KN-2025}.

Our main result is 

\begin{theorem} \label{main-1}
	Let $n, m \in \mathbb{N}$,  $n \ge 4$, and $m \le n-1$. If $K = \GF(3)$ let $n \ge 5$.
	Let $\Omega$ and $\Psi$ be $(m,n-m)$-cyclic conjugacy classes of $\GL(n,K)$. Then  $\Omega \Psi$ contains all nonscalar matrices of $\GL(n,K)$ with determinant $\det \Omega \Psi$.
\end{theorem}

There is a conjecture  attributed to J. G. Thompson \cite[Problem 9.2]{Kourovka-2026}
that a finite simple group
$G$ has conjugacy classes $C$ with covering number 2, i.e. $G = C^2$.
As a  corollary, we obtain the following theorem of Lev  which proves Thompson's conjecture for the simple linear groups.

\begin{corollary} {\rm (\cite[Theorem 3]{Lev-1999})}\label{Lev-1}
	Let $n \ge 3$ if $|K| \le 3$. Then 
	$\PSL(n,K)$ has a conjugacy class with covering number 2.
\end{corollary}

\begin{proof}
	See also \cite{BN-1999} if  $n$ is odd or $K \ne \GF(2)$. 
	Let $\Omega$ be the cyclic conjugacy class of $\GL(n,K)$ with minimal polynomial $\mu(\Omega) = p(x)$, where $p(x)$ is as follows:
	\begin{enumerate}
		\item If $|K| \ge 4$ let $p(x) = (x-1)^{n-2}(x-\delta)(x-\delta^{-1})$, where $\delta \ne \pm 1$.  
		\item If $K = \GF(3)$ and $n = 2t + 6$ let $p(x) = (x^2+1)^t (x^3-x^2-1)(x^3+x-1)$
		\item If $K = \GF(3)$ and $n = 4$ let $p(x) = (x-1)(x+1) (x^2+ x -1)$.
		\item If $K = \GF(3)$ and $n = 2t+1$ let $p(x) = (x-1) (x^2+1)^t$.
		\item If $K = \GF(2)$ and $n = 2t+1$ let $p(x) = (x+1) (x^2+x+1)^t$. 
		\item If $K = \GF(2)$ and $n = 2t + 2$ let $p(x) = (x+1)^2 (x^2+x+1)^t$
	\end{enumerate}
	Then $\Omega$ is a conjugacy class of $\SL(n,K)$. By theorems \ref{BN} and \ref{main-1}, $\Omega^2$ contains all nonscalar matrices of $\SL(n,K)$. And $\Idm_n \in \Omega^2$ as $\Omega = \Omega^{-1}$ except when  $K = \GF(3)$ and $n = 4$. In the exceptional case
	$K = \GF(3), n = 4$, we have $\Omega = -\Omega^{-1}$ so that $-\Idm_n \in \Omega^2$.
\end{proof}

Theorem \ref{main-1}  can  also be used to prove that certain  simple unitary and orthogonal groups have  conjugacy classes of covering number 2. Here we show 

\begin{theorem}  \label{main-2} 
	 Let $h$ be a hyperbolic hermitian form on a finite-dimensional vector space $V$ over a field $K$. Let  $\dim V \ge 4$. If $K \le 9$ let $\dim V \ge 8.$
	Then  $\PSU(V,K, h)$ has a conjugacy class with covering number 2.
\end{theorem}

This improves Bünger's result \cite[Satz 1.5.5]{Bunger-1997}.
In his proof, Bünger uses Lev's theorem  \cite[Theorem 2]{Lev-1994}. So he had to assume that $|K| > 9$.

\begin{remark} 
	If $K = \GF(4)$ and $\dim V =4$, then $\PSU(V,K, h)=\PSU(4,2)$ is isomophic to $\PSp(4,3)$.
	By \cite[Corollary 1.2]{KN-2025}, $\PSp(4,3)$ has a conjugacy class of covering number two.
\end{remark}

Using the computer algebra system GAP \cite{GAP2026}, we can show

\begin{remark} 
	Let $q \in \{2,3\}$.
	Let $\Omega$ be the conjugacy class of $\SUG(6,q)$ with minimal polynomial 
	$(x^3 - x -1)(x^3 +x^2  -1)$. Then
	\begin{enumerate}
		\item $\SUG(6,3) = \Omega^2 \cup -\Idm_6$.
		\item $\SUG(6,2) = \Omega^2 \cup \{\lambda \Idm_6, \lambda^2 \Idm_6\}$, where $\lambda \in \GF(4), \lambda \ne 0, 1$.
	\end{enumerate}
\end{remark}

\section{Proof of  theorem \ref{BN}} 
In \cite{BN-1999}, the authors used singular versions of Sourour's theorem \cite[Theorem 1]{Sourour-1986} and theorem \ref{BN-2} to show that the product of a cyclic similarity class $\Xi$ and a nilpotent cyclic similarity class $\Sigma$ contains all nonnilpotent matrices with nullity one. Here we prove  the weaker result that $\Xi \Sigma$ contains all matrices $0_1 \oplus N$. Then we proceed as in \cite{BN-1999}. Let $\M(n, K)$ denote the set of all $n \times n$-matrices over $K$.

\begin{lemma} \label{cyc1}
	Let $M \in \M(n, K)$ and $Z \in \M(n-1, K)$ be cyclic. Then $M$ is similar to a matrix $P$ with $\PC_{n-1}(P) = Z$.
\end{lemma}

\begin{proof}
		Let
	\[
	M = \left[ \begin{array}{cc} 0 & \Idm_{n-1} \\ \gamma &  d \end{array}\right],
	T = \left[ \begin{array}{cc} \Idm_{n-1} & 0 \\ c &  1 \end{array}\right].
	\] 
	where $c = (c_1, c_2, \dots, c_{n-1})$. Then 
	
	\[
	M^T = \left[ \begin{array}{ccc} 0 & \Idm_{n-2} & \ast \\ c_1 &  \widetilde{c} & \ast \\                  \ast & \ast & \ast \end{array}\right].
	\] 
\end{proof}	

\begin{lemma} \label{cyc1b}
	Let $L \in \GL(n-1, K)$ be lower triangular, and let $M  \in \M(n, K)$ be cyclic. Then $M$ is similar to  matrices 
	\[
	P = \left[ \begin{array}{cc} w  & \beta \\ L &  0 \end{array}\right],
	Q = 	\left[ \begin{array}{cc} 0  & b\\ L&  \delta \end{array}\right].
	\]
\end{lemma}

\begin{lemma} \label{cyc2}
	Let  $\alpha \in K$, and let $D \in \GL(n-1, K)$. 
	Let $\Xi, \Phi \subseteq \M(n, K)$ be cyclic similarity classes of $\M(n, K)$ with $\det \Xi \Phi = \alpha \det D$. Then there exists a row vector $b$ such that 
	\[
	M  = 	\left[ \begin{array}{cc} \alpha  & b \\ 0 &  D \end{array}\right]
	\in \Xi \Phi.
	\]
\end{lemma}

\begin{proof}
	We may assume that $D = LU$, where $L$ is lower and $U$ is upper triangular. By lemma \ref{cyc1b}, there exist matrices
	\[
	P = 	\left[ \begin{array}{cc} w  & \beta \\ L &  0 \end{array}\right]
	\in \Xi,
	Q = 	\left[ \begin{array}{cc} 0  & U \\ \gamma &  v \end{array}\right]
	\in \Phi.
	\]
	Then 
	\[
	PQ = 	\left[ \begin{array}{cc} \beta \gamma & wU \\ 0 &  LU \end{array}\right].
	\] 
\end{proof}	

We prove theorem \ref{BN}:
Let $D \oplus \delta I_1 \in \Psi$.
Let $M \in \GL(n, K)$ be nonscalar. It follows easily from lemma \ref{cyc1}
that $M$ is similar to a matrix
\[
P = 	\left[ \begin{array}{cc} S  & q \\ p &  \rho \end{array}\right].
\]
where $S = 0_1 \oplus Q$ for some matrix $Q \in \GL(n-2, K)$. By lemma \ref{cyc2}, we may assume that $S = N D$, where $N = \Jordan_{n-1}'$. Now
$(pD^{-1})_{n-1} \ne 0$: Otherwise, $P (D^{-1} \oplus I_1)$ would be singular.
By lemma \ref{cyc1b}, $\Omega$ contains a matrix
\[
C = 	\left[ \begin{array}{cc} N  & b \\ pD^{-1} &  \gamma \end{array}\right].
\]
Let 	
\[
C^{-1} P = 	\left[ \begin{array}{cc} X  & v \\ u &  \delta \end{array}\right].
\]
Then $NX +bu = S = ND$ and $pD^{-1}X + \gamma u = p$. Hence $u = 0$ as $b_1 \ne 0$. Then $X = D$ so that $C^{-1} P \in \Psi$.


\section{Matrices with minimal rank one} 
\mbox{}

\begin{definition}
	{\rm	Let $\mr(M) = \min \{\rank(M-\lambda \Idm_n); \lambda \in \overline{K}\}$, 
		where $\overline{K}$ is an algebraic closure of $K$, denote the 
		the {\em minimal  rank} $\mr(M)$ of a matrix $M \in \M(n,K)$. }
\end{definition}

Clearly, $\mr(M)$ is the number of trivial invariant factors of $M$.

\begin{lemma} \label{simple-lemma}
Let $n \ge 3$ if $|K| = 2$.
Let $\Omega$ and $\Psi$ be cyclic conjugacy classes of $\GL(n,K)$. 	Let $M \in \GL(n,K)$ 
with $\mr(M) = 1$ and $\det M = \det \Omega \Psi$. Then $M \in \Omega \Psi$ if and only if one of the following holds
	\begin{enumerate}
		\item $M$ is nonprimary; 
		\item  $\Omega$ or $\Psi$ is not irreducible.
	\end{enumerate}
\end{lemma}

\begin{proof}
1: 
	\[
	\left[ \begin{array}{cc} \Idm_{n-1}  & 0 \\ d + \gamma a &   \gamma \beta \end{array}
	\right] 
	= 
	\left[ \begin{array}{cc} 0  & \Idm_{n-1} \\ \gamma &  d \end{array}
	\right]
	\left[ \begin{array}{cc} a  &  \beta \\ \Idm_{n-1} &  0 \end{array} \right].
	\]
2: So let $P$ be primary. We may assume that $P$ is a transvection. As just seen, $P \in \Omega \Psi$ or $\Idm_n \in \Omega \Psi$. So we may assume that $\Psi = \Omega^{-1}$.  First let $\Omega$ be $(m,n-m)$-cyclic. Then $P$ is similar to a matrix
	\[
	\left[ \begin{array}{cc} \Idm_m  & N \\ 0 &   \Idm_{n-m} \end{array}
	\right] = 
	\left[ \begin{array}{cc} Q  & 0 \\ 0 &  R \end{array}
	\right]
	\left[ \begin{array}{cc} Q^{-1}  & Q^{-1}N \\ 0 &  R^{-1} \end{array}
	\right]
	\in \Omega \Psi,
	\]
where $Q \oplus R \in \Omega$.
	
So let $\Omega$ be primary. Let $W = \Jordan_k(p) \in \Omega$ be the upper Jacobson (hypercompanion) matrix with minimal polynomial $p^k$, where $p \in K[x]$ is irreducible. Let $N = \Jordan_n(0):= \Jordan_n(x)$, and let $T$ be the transvection  $\Idm_n + N^{n-1}$. Then $WT$ is similar to $W$. 
	
Finally, let $\Omega$ be irreducible. Suppose that  $WP \in \Omega$ for some $W \in \Omega$. So let $WP =X^{-1}WX$. Then $XWP = WX$ and $XW -WX = WXP^{-1} - WX$ has rank one. By Uhlig's rank one lemma \cite[Lemma 2]{Uhlig-1979}, $XW - WX = 0$, a contradiction.
\end{proof}

We give a short proof of Uhlig's rank one lemma

\begin{lemma} \label{Uhlig-rank-one-lemma}	Let $P \in \GL(n,K)$ be irreducible. Let $H = XP -PX$ for some $X \in \GL(n, K)$. If $\rank H \le 1$, then $H = 0$.
\end{lemma}

\begin{proof}
The matrices $HP^j$ are nilpotent of rank at most one as they have zero trace.
Hence $HP^jH = 0$. Now if $u \in K^n$, then $\langle uH, uHP, \dots, uHP^{n-1}\rangle \le \ker H$. Hence $uH = 0$.
\end{proof}


\section{Proof of theorem \ref{main-1}}

Let $\Omega = \Omega_1 \oplus \Omega_2, \Psi = \Psi_1 \oplus \Psi_2$, where $\Omega_1, \Psi_1 \subseteq \GL(m, K)$. By thorem \ref{BN}, we may asssume that $m \ne 1, n-1$. And by theorem \ref{BN-2} and lemma \ref{simple-lemma}, we only have to show that $\Omega \Psi$ contains all noncyclic, primary matrices with minimal rank at least 2.
 For the proof we use a block LU-decomposition. 
 
 \begin{definition}
 	Let 
 	\[
 	M = 	\left[ \begin{array}{lr} A  & B\\ C &  D  \end{array}	\right],
 	\]
 	where $A \in \M(m, K)$. We say that the matrix $A$ is the principle corner $\PC_m(M)$
 	of $M$. If $A$ is nonsingular, then  the matrix $\SC_m(M) :=  D-CA^{-1}B$ is called the Schur complement of $A$ in $M$.
 \end{definition}	
 
 We remark some useful properties of Schur complements.
 
 \begin{lemma} \label{Schur-0}
 	Let 
 	\[
 	M = 	\left[ \begin{array}{lr} A  & B\\ C &  D  \end{array}
 	\right], 
 	X = 	\left[ \begin{array}{lr} A^{-1}B  & \Idm_m\\ \Idm_{n-m} &  0  \end{array}
 	\right],
 	Q = 	\left[ \begin{array}{lr} Y  & 0\\ 0 &  Z  \end{array}
 	\right], 
 	\]
 	where $M \in \GL(n, K)$, $A, Y \in \GL(m, K)$,  and $Z \in \GL(n-m, K)$.
 	\begin{enumerate}
 		\item $\PC_{n-m}(M^X) = \SC_m(M)$ and $\SC_{n-m}(M^X) = A$.
 		\item $\PC_m(M^Q) = \PC_m(M)^Y, \SC_m(M^Q) = \SC_m(M)^Z$.
 \end{enumerate}
 \end{lemma}

\begin{lemma} \label{Schur-1}
Let 
	\[
	M = 	\left[ \begin{array}{lr} A  & B\\ C &  D  \end{array}	\right]
	\]
	with a non singular principal corner $A= \PC_m(M)$. Let $S = \SC_m(M)$ be the Schur complement of $A$ in $M$. If $A = A_1 A_2$ and $S = S_1 S_2$, then $M$ has a decomposition
	\[
	M = 	\left[ \begin{array}{lr} A  & B\\ C &  D  \end{array}	\right]
	= \left[ \begin{array}{cc} A_1  & 0\\ CA_2^{-1} &  S_1  \end{array}	\right]
	\left[ \begin{array}{cc} A_2  & A_1^{-1}B\\ 0 &  S_2  \end{array}	\right].
	\]
\end{lemma}

We say that a matrix $P \in \GL(n, K)$ is good\footnote{Halmos \cite{Halmos-1991} probably would prefer to call the pair $(\Omega, \Psi)$ good} (for our proof) if $P$ is similar to a matrix $Q$ with $\PC_m(Q) \in \Omega_1 \Psi_1$ and $\SC_m(Q) \in \Omega_2 \Psi_2$. 
As just shown, $\Omega \Psi$ contains all good matrices.

So if we can achieve that $\PC_m(Q)$ and  $\SC_m(Q)$ are nonprimary and $\PC_m(Q)$ has prescribed determinant, then $Q$ is good by theorem \ref{BN-2}. We show that this is possible if $K$ or $m, n-m$, and the minimal rank of $M$ are not too small.

It is easy to see that $M \in \M(n, K)$ is similar to a 
matrix $P$ such that $\PC_m(P)$ can be any prescribed matrix if $m \le \mr(M), \frac{n}{2}$.

This follows immediately from the following result of Sourour.
A subspace $T$ of a vector space $V$ is called $\varphi$-antiinvariant if  $\varphi$ is a linear transformtion  of $V$ and $T \cap T \varphi = 0$.

\begin{lemma} \label{Sourour}
Let $V$ be a vector space with finite dimension $\dim V$. Let $\varphi \in \GL(V)$. Then 
$V$ has a $\varphi$-antiinvariant subspace of dimension $m$ if and only if $m \le \mr(\varphi), \frac{\dim V}{2}$.
\end{lemma}

For a proof see \cite[Theorem]{Sourour-1986b} or \cite[Theorem]{BH-1984}.

\begin{lemma} \label{Schur-1a}
Let $M \in \GL(n,K)$. Let $m \in \mathbb{N}$, and let $m \le \mr(M), \frac{n}{2}$. Then $M$ is similar to a matrix
	\[
	P = \left[ \begin{array}{ccc} 0  & 0 & \Idm_m\\ 0 &  B & X \\ C & D & Y \end{array}
	\right],
	\]
where $B \in \GL(n-2m, K), C \in \GL(m, K)$. 
\end{lemma}

\begin{proof}
By \ref{Sourour}, $M$ is similar to a matrix
	\[
	P = \left[ \begin{array}{ccc} 0  & 0 & \Idm_m\\ A &  B & X \\ C & D & Y \end{array}
	\right],
	\]
where $B \in \M(n-2m, K), C \in \M(m, K)$. 
	\[
	P \sim  
	\left[ \begin{array}{ccc} \Idm_m  & 0 & 0\\ -S &  \Idm_{n-2m} & 0 \\ 0 & 0 &  \Idm_m\end{array} \right]
	\left[ \begin{array}{ccc} 0  & 0 & \Idm_m\\ A &  B & X \\ C & D & Y \end{array} \right]
	\left[ \begin{array}{ccc} \Idm_m  & 0 & 0\\ S &  \Idm_{n-2m} & 0 \\ 0 & 0 &  \Idm_m\end{array} \right]
	\]
		\[
	= \left[ \begin{array}{ccc} 0  & 0 & \Idm_m\\ A +BS &  B & X-S \\ C+DS & D & Y \end{array} \right].
	\]
	So we can achieve that $C$ is nonsingular. Further, 
	\[
	P \sim  
	\left[ \begin{array}{ccc} \Idm_m  & 0 & 0\\ 0 &  \Idm_{n-2m} & -T \\ 0 & 0 &  \Idm_m\end{array} \right]
	\left[ \begin{array}{ccc} 0  & 0 & \Idm_m\\ A &  B & X \\ C & D & Y \end{array} \right]
	\left[ \begin{array}{ccc} \Idm_m  & 0 & 0\\ S &  \Idm_{n-2m} & T \\ 0 & 0 &  \Idm_m\end{array} \right]
	\]
	\[
	= \left[ \begin{array}{ccc} 0  & 0 & \Idm_m\\ A -TC &  B-TD & X+BT -TY -TDT \\ C+DS & D & Y +DT \end{array} \right].
	\]
	Putting $T = AC^{-1}$, we are done.
\end{proof}

In fact, \ref{Sourour} and \ref{Schur-1a} are trivial  consequences of the
interlacing theorem of R. C. Thompson \cite[Theorem 6]{RCThompson-1979} and de E. M. de S\'a \cite[Theorem 5.4]{deSa-1979}:

\begin{lemma} \label{INTERLAC}
Let $A \in \M(m,K)$, and let $M \in \GL(m+q,K)$. Then $M$ is similar to a matrix $Q$ such that
$\PC_m(Q) = A$ if and only if the invariant factors of $M$ and $A$ satisfy the following inequalities $i_k(M)| i_k(A)| i_{k+2q}(M)$.
\end{lemma}

\begin{lemma} \label{Schur-1b}
Let $M \in \GL(n,K)$. Let $m \in \mathbb{N}$, and let  $m \le \frac{n}{2}, \mr(M)$. There exist matrices  $B \in \GL(n-2m, K), C \in \GL(m, K), D \in \M(m, n-2m)$ such that $M$ is similar to matrices $M_S$ with $\PC_m(M_S) = S \in \GL(m, K)$ and 
\[
	\SC_m(M_S) =	\left[ \begin{array}{cc} B  & 0\\ D &  -CS^{-1}  \end{array}
	\right].
\]
\end{lemma}

\begin{proof}
Using \ref{Schur-1a}, we see that $M$ is similar to  matrices
\[
M_S =	\left[ \begin{array}{ccc} \Idm_m  & 0 & 0\\ 0 &  \Idm_k & 0 \\ -S & 0 & \Idm_m \end{array}\right]
\left[ \begin{array}{ccc} 0  & 0 & \Idm_m\\ 0 &  B & X \\ C & D & Y \end{array}
\right]
\left[ \begin{array}{ccc} \Idm_m  & 0 & 0\\ 0 &  \Idm_k & 0 \\ S & 0 & \Idm_m \end{array}
\right]
\]
\[
= \left[ \begin{array}{ccc} S  & 0 & \Idm_m\\ XS &  B & X \\ C + (Y-S)S & D & Y-CS \end{array}\right].
\]
\end{proof}

\begin{lemma} \label{Schur-2}
Let $\delta \in K^*$. Let $n \ge 4$. Let $M \in \GL(n,K)$. Let $m \in \mathbb{N}$, 
and let  $2 \le m \le \frac{n}{2}, \mr(M)$.  If $|K| \le 3$ let $m, \mr(M) \ge 3$. Then 
$M$ is similar to a matrix $P$ such that $\PC_m(P)$ and $\SC_m(P)$ are nonprimary, and $\det \PC_m(P) = \delta$.
\end{lemma}

\begin{proof}
Clearly, $M$ is similar to a matrix $\lambda \Idm_p \oplus Q$, where $\mr(Q) = \mr(M) \ge \frac{n-p}{2}$. So we may assume that $m \le \mr(M)$.
Let $M_S$ as in the proof of lemma \ref{Schur-1b}.
The case that $C$ is scalar is trivial. So let $C$ be nonscalar. If $K \ne \GF(2)$ we can use Sourour's theorem \cite[Theorem 1]{Sourour-1986}: $C$ is similar to the product of a  lower triangular matrix with prescribed diagonal elements and an  upper triangular matrix with $m-1$  prescribed diagonal elements. It remains to consider the case $K = \GF(2)$. Now $C$ is similar to a matrix
\[
	\left[ \begin{array}{cc} \Idm_2  & b\\ c &  E  \end{array} \right]
	= \left[ \begin{array}{cc} X  & 0\\ c X &  L  \end{array} \right]
	\left[ \begin{array}{cc} X^{-1}  & X^{-1} b\\ 0 &  U  \end{array} \right].
\]
where $X^2 + X + \Idm_2 = 0$, and $E + cb = LU$ is the product of 2 unipotent matrices $L$ and $U$.
\end{proof}

So it remains to consider the cases $|K| \le 3, m = 2$ and $|K| \le 3, \mr(M) = 2$.

This requires several additional  calculations. This is not very surprising. In many matrix factorization theorems the authors assume that the ground field has at least 4 elements.\footnote{See also R. C. Thompson's remark \url{https://zbmath.org/0548.15010}}


\subsection{The case $K = \GF(3), n \ge 5$}
\mbox{}

Let  $\mathbb{S}_2, \mathbb{X}_2, \mathbb{U}_2, \mathbb{Y}_2$ be the conjugacy classes of $\GL(2,3)$ with $\mu(\mathbb{S}_2) =x^2-1, \mu(\mathbb{X}_2) = x^2+1, \mu(\mathbb{U}_2) =(x-1)^2, \mu(\mathbb{Y}_2) = x^2 +x-1$.

\begin{remark} \label{Schur-3}
	\begin{enumerate}
		\item[]
		\item $\mathbb{S}_2 \mathbb{S}_2 = \SL(2,3)$.
		\item $\mathbb{X}_2 \mathbb{X}_2 = -\Idm_2 \cup \Idm_2 \cup \mathbb{X}_2$. 
		\item $\mathbb{U}_2 \mathbb{U}_2 = \mathbb{U}_2 \cup -\mathbb{U}_2 \cup \mathbb{X}_2 \cup \Idm_2$.
		\item $\mathbb{X}_2 \mathbb{U}_2 = \mathbb{U}_2 \cup -\mathbb{U}_2$.
		\item $\mathbb{Y}_2 \mathbb{Y}_2 = -\Idm_2 \cup \mathbb{X}_2 \cup \mathbb{U}_2$.
		\item $\mathbb{Y}_2 \mathbb{Y}_2^{-1} = \Idm_2 \cup \mathbb{X}_2 \cup -\mathbb{U}_2$.
		\item $\mathbb{U}_2 \mathbb{Y}_2 = -\mathbb{Y}_2 \cup \mathbb{S}_2$.
		\item $\mathbb{X}_2 \mathbb{Y}_2 = -\mathbb{Y}_2 \cup \mathbb{Y}_2 \cup \mathbb{S}_2$.
	\end{enumerate}
\end{remark}

\begin{lemma} \label{Schur-4}
	Let $M \in \GL(n,3)$ with $\mr(M) \ge 2$.  Let $n \ge 5$. Then 
	$M$ is similar to a matrix $Q$ with $\PC_2(Q) \in \mathbb{S}_2$ such that $\SC_2(Q)$ is nonprimary. Further, $M$ is similar to  matrices $P_1$ and $P_2$, where 
	\begin{enumerate}
		\item $\PC_2(P_1) \in \mathbb{U}_2,  \PC_2(P_2) \in \mathbb{X}$ or
		\item $\PC_2(P_1) \in -\mathbb{U}_2,  \PC_2(P_2) = -\Idm_2$.
	\end{enumerate}
	and  $\SC_2(P_1)$ and $\SC_2(P_2)$  are nonprimary. 
\end{lemma}

\begin{proof}
	As above,  $M$ is similar to  matrices $M_S$ with $\PC_2(M_S) = S$ and  
	\[
	\SC_2(M_S) =	\left[ \begin{array}{cc} B  & 0\\ D &  -CS^{-1}  \end{array}
	\right].
	\]
	First let $C$ be nonscalar. By \ref{Schur-3}, the product of 2 cyclic conjugacy classes in $\GL(2,3)$ contains at least 2 coprimary matrices.
	There exist matrices $S \in \mathbb{S}_2, U \in \mathbb{U}_2, X \in \mathbb{X}_2$ such that $\PC_2(M_S), \PC_2(M_U), \PC_2(M_X)$ are nonprimary.
	
	So let $C = \epsilon \Idm_2$ be scalar. Clearly, $\PC_2(M_S)$ is nonscalar for $S \in \mathbb{S}_2$. If $\epsilon B$ is unipotent, then $\PC_2(M_U)$ and $\PC_2(M_X)$ are nonprimary for all $U \in \mathbb{U}_2, X \in \mathbb{X}_2$. 
	If $\epsilon B$ is not unipotent and $-S$ is unipotent, then $\PC_2(M_S)$ is nonprimary. 
\end{proof}

\begin{corollary} \label{Schur-5}
	Let $P \in \GL(n,3)$ with $\mr(P) = 2$.  Let $3 \le m \le n-3$, and let $\epsilon = \pm 1$. Then $P$ is similar to a matrix $Q$ such that $\PC_m(Q)$ and $\SC_m(Q)$ are not primary, and $\det \PC_m(Q) = \epsilon$.
\end{corollary}

\begin{proof}
We may assume that $P = \Idm_{n-5} \oplus M$, where $\mr(M) = 2$. It follows from lemma \ref{Schur-4} that 
$P$ is  similar to a matrices $P_1$ and $P_2$, where $\PC_m(P_1) =-\Idm_1 \oplus \Idm_{m-1}$, $\PC_m(P_2) \in -\mathbb{U}_2 \oplus \Idm_{m-2} \cup \mathbb{X}_2 \oplus \Idm_{m-2}$, and  $\SC_m(P_1)$ and $\SC_2(P_2)$ are nonprimary. 
\end{proof}

It follows from remark \ref{Schur-3} and theorem \ref{BN-2} that the matrices in $\GL(n,3)$ are good.


\subsection{The case $K=\GF(2)$}
\mbox{}

First we consider the case $m = 2$.

Let  
$\mathbb{X}_2, \mathbb{U}_2$ be the conjugacy classes of $\GL(2,2)$ with $\mu(\mathbb{X}_2) = x^2+x+1, \mu(\mathbb{U}_2) = x^2+1$. Let $\Jordan_k(1) \in \GL(k,2)$ denote  the upper Jordan matrix with minimal polynomial $(x+1)^k$. Let $\partial_k(P) = \partial i_k(P)$ denote the degree of the $k$th invariant factor of $P$.

\begin{remark} \label{Schur-6}
	\begin{enumerate}
		\item[]
		\item $\mathbb{X}_2 \mathbb{X}_2 = \Idm_2 \cup \mathbb{X}_2$. 
		\item $\mathbb{U}_2 \mathbb{U}_2 =  \Idm_2 \cup \mathbb{X}_2$.
		\item $\mathbb{X}_2 \mathbb{U}_2 = \mathbb{U}_2$.
	\end{enumerate}
\end{remark}

For the following lemma,  we need the full strength of the interlacing theorem of Thompson and de S\'a.
\begin{lemma} \label{Schur-6a}
	Let $M \in \GL(n, 2)$. Assume that $\partial_n(M) + \partial_{n-1}(M) \ge 6$. Then
	$M$ is similar to matrices $M_1, M_2$, where
	\begin{enumerate}
		\item $\PC_2(M_1 ) \in \mathbb{U}_2$, $\SC_2(M_1)$ is nonprimary,
		\item $\PC_2(M_2) \in \mathbb{X} \cup \Idm_2$, $\SC_2(M_2)$ is nonprimary.
	\end{enumerate}
\end{lemma}

\begin{proof}
	We may assume that $\mr(M) \ge n-2$. As before, $M$ is similar to matrices  
	\[
	M_S =	\left[ \begin{array}{ccc} \Idm_m  & 0 & 0\\ 0 &  \Idm_k & 0 \\ -S & 0 & \Idm_m \end{array}\right]
	\left[ \begin{array}{ccc} 0  & 0 & \Idm_m\\ 0 &  B & X \\ C & D & Y \end{array}
	\right]
	\left[ \begin{array}{ccc} \Idm_m  & 0 & 0\\ 0 &  \Idm_k & 0 \\ S & 0 & \Idm_m \end{array}
	\right]
	\]
	\[
	= \left[ \begin{array}{ccc} S  & 0 & \Idm_m\\ XS &  B & X \\ C + (Y-S)S & D & Y-CS \end{array}\right]
	\]
	with $\PC_m(M_S) = S \in \GL(m, K)$ and 
	\[
	\SC_m(M_S) =	\left[ \begin{array}{cc} B  & 0\\ D &  -CS^{-1}  \end{array}
	\right].
	\]
	
	By the interlacing theorem \ref{INTERLAC}, we may assume that $B$ is irreducible.
	
	The case $n \ge 7$ is trivial. So let $n = 6$.  Let $U \in  \mathbb{U}_2$. 
	\begin{enumerate}
		\item If $C = \Idm_2$, then $\SC_2(M_U) \sim B \oplus U$ and $\SC_2(P_I) \sim B \oplus \Idm_2$.
		\item If $C  \sim U$, then $\SC_2(M_C) \sim B \oplus \Idm_2$ and $\SC_2(P_I) \sim B \oplus \mathbb{U}_2$.
		\item If $C  \sim B$, then $\SC_2(M_C) \sim  B \oplus \Idm_2$ and $\SC_2(P_U) \sim  B \oplus \mathbb{U}_2$.
	\end{enumerate}
\end{proof}

\begin{lemma} \label{Schur-6b}
	Let $P  \in \mathbb{X}_2 \oplus \mathbb{X}_2$. Then $P$ is similar to a matrices $P_1$ with $\PC_2(P_1) = \SC_2(P_1) = \Idm_2$ and $P_2$ with $\PC_2(P_2),  \SC_2(P_2) \in \mathbb{U}_2$.
\end{lemma}

\begin{proof}
	Let $U \in  \mathbb{U}_2$. Let
	\[
	P_1 =	\left[ \begin{array}{cc} \Idm_2 & \Idm_2   \\ \Idm_2 &  0 \end{array} \right],
	P_2 =	\left[ \begin{array}{cc}  U  &  \Idm_2 \\ U &  U + \Idm_2 \end{array} \right].
	\]
\end{proof}

\begin{corollary} 
	Let $n \ge 6$.  Then all transformations of $\GL(n, 2)$ with minimal polynomial
	$x^2+x+1$ are good.
\end{corollary}

So for the rest of the proof  of theorem \ref{main-1}, we are left with unipotent matrices.


\subsubsection{The case $K=\GF(2), n \in {4,5}$}

\begin{lemma} \label{Schur-7}
	Let $P \in \GL(4,2)$ be unipotent with  $\mr(P) = 2$. Then $P$ is
	similar to a matrix $Q$ with $\PC_2(Q), \SC_2(Q) \in \mathbb{U}_2$.
\end{lemma}

\begin{proof}
	$P$ is similar to a matrix
	\[
	\left[ \begin{array}{cc}  \Jordan_2(1)  & B \\ 0 &\Jordan_2(1)  \end{array} \right].
	\]
\end{proof}

\begin{corollary} \label{Schur-8}
	Let $n = 4$. Then all primary matrices of $\GL(4,2)$ with minimal rank 2 are good.
\end{corollary}

\begin{lemma} \label{Schur-9}
	Let $P \in \mathbb{U}_2 \oplus \mathbb{U}_2$. Then $P$ is
	similar to matrices $P_1, P_2$ with
	\begin{enumerate}
		\item $\PC_2(P_1) \in \mathbb{U}_2, \SC_2(P_1) = \Idm_2$,
		\item $\PC_2(P_2), \SC_2(P_2) \in \mathbb{X}_2$.
	\end{enumerate}
\end{lemma}

\begin{proof}
	Let  $X \in \mathbb{X}_2, U \in \mathbb{U}_2$ and
	\[
	P_1 = \left[ \begin{array}{cc}   U & U + \Idm_2\\ U + \Idm_2 &  \Idm_2 \end{array} \right],
	P_2 = \left[ \begin{array}{cc} X  &  \Idm_2\\ X &  X \end{array} \right].
	\]
\end{proof}

\begin{lemma} \label{Schur-9a}
	Let $P \in \GL(4,2)$ be similar to  $\Idm_1 \oplus \Jordan_3(1)$. Then $P$ is
	similar to matrices $P_1, P_2$ with
	\begin{enumerate}
		\item $\PC_2(P_1), \SC_2(P_1) \in \mathbb{U}_2$,
		\item $\PC_2(P_2) \in \mathbb{U}_2, \SC_2(P_2) \in \mathbb{X}_2$.
	\end{enumerate}
\end{lemma}

\begin{proof}
	Let $U = \Jordan_2(1)$ and  $J = U + \Idm_2$.
	Let 
	\[
	P_1 = \left[ \begin{array}{cc}  U  &  J'J\\ 0 &  U \end{array} \right],
	P_2 = \left[ \begin{array}{cc} U  &  U'\\ J'J &  J \end{array} \right].
	\]
\end{proof}

\begin{lemma}  \label{Schur-10}	
	Let $P \in \GL(5,2)$ be similar to be similar to $\Jordan_2(1) \oplus \Jordan_3(1)$ or $\Idm_1 \oplus \Jordan_4(1)$. Then $P$ is similar to matrices $P_1, P_2$, where
	\begin{enumerate}
		\item $\PC_2(P_1) \in \mathbb{U}_2$ and  $\SC_2(P_1) \in \mathbb{X}_2 \oplus \Idm_1$,
		\item $\PC_2(P_2) \in \mathbb{X}_2$ and  $\SC_2(P_2) \in \mathbb{X}_2 \oplus \Idm_1$.
	\end{enumerate}
\end{lemma}

\begin{proof}
 $P$ is similar to matrices
	\[
P_1 = \left[ \begin{array}{cc} A  & b \\ 0 &  1 \end{array} \right],
P_2 = \left[ \begin{array}{cc}  \widetilde{A}  & \widetilde{b} \\ 0 &  1 \end{array} \right],
\]
where $A = \Jordan_3(1) \oplus \Idm_1$ and $\widetilde{A} =  \Jordan_2(1) \oplus \Jordan_2(1)$.
 Apply \ref{Schur-9a} to $A$ and  \ref{Schur-9} to $\widetilde{A}$.
\end{proof}

\begin{lemma}    \label{Schur-11}
Let $P \in \GL(5,2)$ be similar to  $\Idm_2 \oplus \Jordan_3(1)$. Then $P$ is similar to matrices $P_1, P_2, P_3$, where
	\begin{enumerate}
		\item $\PC_2(P_1) = \Idm_2$, $\SC_2(P_1) = \Jordan_3(1)$,
		\item $\PC_2(P_2) = \Idm_2$, $\SC_2(P_2) = \Jordan_2(1) \oplus \Idm_1$,
		\item $\PC_2(P_3) \in \mathbb{U}_2$ and  $\SC_2(P_3) \in \mathbb{X}_2 \oplus \Idm_1$.
	\end{enumerate}
\end{lemma}

\begin{proof}
Clearly, $P$ is similar to $P_1$. Using \ref{cyc1}, we see that $P$ is similar to $P_2$.
And $P$ is similar to a matrix
		\[
	\left[ \begin{array}{cc}  A & b \\ 0 &  1 \end{array} \right],
	\]
where $A =   \Jordan_3(1) \oplus \Idm_1$. Applying lemma \ref{Schur-9a} to $A$, we see that 
$P$ is similar to a matrix $P_3$. 
\end{proof}

\begin{lemma}  \label{Schur-12}
Let $P \in \GL(5,2)$ be similar to $\Idm_1 \oplus \Jordan_2(1) \oplus \Jordan_2(1)$. Then
$P$ is similar to matrices $P_1, P_2, P_3$, where
\begin{enumerate}
	\item $\PC_2(P_1) \in \mathbb{X}_2$ and  $\SC_2(P_1) \in \mathbb{X}_2 \oplus \Idm_1$,
	\item $\PC_2(P_2) \in \mathbb{U}_2$ and  $\SC_2(P_2) \in \mathbb{U}_2 \oplus \Idm_1$,
	\item $\PC_2(P_3) \in \mathbb{U}_2$ and  $\SC_2(P_3) =  \Idm_3$.
\end{enumerate}
\end{lemma}

\begin{proof}
Clearly, $P \sim P_2$. By lemma \ref{Schur-9}, $P \sim P_1$ and  $P_3$.
\end{proof}

\begin{corollary} 
Let $n = 5$. Then all primary matrices  $M \in \GL(5,2)$ with minimal rank 2 or 3 are good.
\end{corollary}

\begin{proof}
Clearly, $M$ must be unipotent. Apply \ref{Schur-10}, \ref{Schur-11}, and \ref{Schur-12}  together with \ref{Schur-6} and \ref{BN-2}.
\end{proof}

\subsubsection{The case $K=\GF(2), n \ge 6$}

\begin{lemma} \label{Schur-13}
Let $n \ge 6$. Let $P  \in \GL(n,2)$ be unipotent but not involutory. Then $P$ is good.
\end{lemma}

\begin{proof}
If $\partial_n(P) + \partial_{n-1}(P) \ge 6$ apply lemma \ref{Schur-6a}.
Assume next that   $\partial_n(P) + \partial_{n-1}(P)= 5$. Then $P$ is similar to $Q \oplus U$, where $Q \in \GL(5, 2)$, and $U$ is involutory. By lemma \ref{Schur-10}, $P$ is similar to matrices $P_1$ and $P_2$ , where
	\begin{enumerate}
	\item $\PC_2(P_1) \in \mathbb{U}_2$ and  $\SC_2(P_1) \sim  X \oplus U$,
	\item $\PC_2(P_2) \in \mathbb{X}_2$ and  $\SC_2(P_2) \sim X  \oplus U$.
\end{enumerate}
	So let  $\partial_n(P) = 3,  \partial_{n-1}(P) = 1$. By lemma \ref{Schur-11},
	$P$ is similar to matrices $P_1, P_2$, and $P_3$, where
	\begin{enumerate}
		\item $\PC_2(P_1) = \Idm_2$ and  $\SC_2(P_1) \sim  \Jordan_3(1) \oplus \Idm_{n-5}$,
		\item $\PC_2(P_2) = \Idm_2$ and  $\SC_2(P_2) \sim  \Jordan_2(1) \oplus \Idm_{n-4}$,
		\item $\PC_2(P_3) \in \mathbb{U}_2$ and  $\SC_2(P_3) \sim X  \oplus \Idm_{n-2}$.
	\end{enumerate}
\end{proof}

\begin{lemma} \label{Schur-14}
	Let $P \in  \mathbb{U}_2 \oplus \mathbb{U}_2 \oplus \mathbb{U}_2$.
	Then $P$ is similar to a matrix $Q$ with $\PC_2(Q) \in \mathbb{U}_2$ and $\SC_2(Q) \sim \Jordan_3(1) \oplus \Idm_1$.
\end{lemma}

\begin{proof}
$P$ is similar to
\[
\left[ \begin{array}{cccccc} 
	0 & 1 & 0 & 0 & 1 & 0 \\
	1 & 0 & 0 & 0 & 1 & 0 \\
	0 & 1 & 1 & 1 & 0 & 0 \\
	1 & 1 & 0 & 1 & 1 & 0 \\
	0 & 0 & 0 & 0 & 1 & 0 \\
	1 & 1 & 0 & 0 & 0 & 1 
\end{array} \right].
\]
\end{proof}

\begin{lemma} \label{Schur-15}
Let $n \ge 6$. Then  all involutory  matrices 
$P \in \GL(n, 2)$  with $\mr(P) \ge 2$ are good.
\end{lemma}

\begin{proof}
If $\mr(P) = 2$, then it follows from \ref{Schur-12} that $P$ is  good. So let $P = U_2 \oplus U_2 \oplus U_2 \oplus U$, where $U_2\in \mathbb{U}_2$.
By \ref{Schur-9}, \ref{Schur-7}, and \ref{Schur-14}, $P$ is similar to a matrices $P_1, P_2, P_3, P_4$, where 
\begin{enumerate}
	\item $\PC_2(P_1) \in   \mathbb{X}_2, \SC_2(P_1) \sim  X \oplus U_2 \oplus U$,
	\item $\PC_2(P_2) \in   \mathbb{U}_2, \SC_2(P_2) \sim  \Idm_2 \oplus U_2 \oplus U$,
	\item $\PC_2(P_3) \in   \mathbb{U}_2, \SC_2(P_3) \sim  U_2 \oplus U_2 \oplus U$,
	\item $\PC_2(P_4) \in   \mathbb{U}_2, \SC_2(P_4) \sim  \Idm_1 \oplus \Jordan_3(1) \oplus U$.
\end{enumerate}
	
Let $\widetilde{U} = \Idm_1 \oplus  U$. By \ref{cyc2}, there exists a row vector $b$ such 
\[
U_b :=  \left[ \begin{array}{cc} 1 & b \\ 0 &  \widetilde{U}\end{array}
\right] \in \Omega_2 \Psi_2.
\]
It is easy to see (e.g. consider the Weyr characteristic of $U_b$: $ \dim \ker U -1 \le \dim \ker U_b \le \dim \ker U, \dim \ker U_b^2 \le 1$) that $U_b$ is similar to one of the matrices  $\SC_2(P_2), \SC_2(P_3)$ or $\SC_2(P_4)$.
\end{proof}

Finally, we have to consider the case $3 \le m \le n-3$.

\begin{lemma} \label{Schur-16}
Let $3 \le m \le n-3$.
Then all unipotent matrices $P \in \GL(n,2)$ with $\mr(P) = 2$ are good.
\end{lemma}

\begin{proof}
By \ref{Schur-7}, $P$ is similar to  matrices $P_1, P_2, P_3, P_4$ such that 
\begin{enumerate}
	\item $\PC_m(P_1) = \Idm_m,  = \SC_m(P_1) = \Idm_{n-m}$,
	\item $\PC_m(P_2) \sim \Idm_{m-2} \oplus \Jordan_2(1)$, $\SC_m(P_2) \sim \Idm_{n-m-2}      \oplus \Jordan_2(1)$, 
	\item $\PC_m(P_3) = \Idm_m$, $\SC_m(P_3) \sim \Idm_{n-m-2} \oplus \Jordan_2(1)$, 
	\item $\PC_m(P_4) \sim \Idm_{m-2} \oplus \Jordan_2(1)$, $\SC_m(P_4) = \Idm_{n-m}$. 
	\end{enumerate}
Apply lemma \ref{cyc2}.
\end{proof}

\section{Proof of theorem \ref{main-2}}

Let $\UG(2n, K)$ be the set of all matrices 
\[
P =   \left (\begin{array} {cc} A & B\\ C &  D \end{array} \right ) \in \GL(2n, K)
\]
with $P H\overline{P}' = H$, where 
\[
H =   \left (\begin{array} {cc} 0 & \Idm_n\\ \Idm_n &  0 \end{array} \right ).
\]
A direct computation shows that if A is nonsingular, then
$\SC_n(A) = \overline{A^+}$.
We call a unitary transformation $\varphi$ strictly hyperbolic  if its minimal polynomial is of the form $p(x) \overline{p}^*(x)$, where $p(x)$ is prime to its conjugate reciprocal $\overline{p}^*(x)$. If $\varphi$ is strictly hyperbolic, then 
$\ker p(\varphi)$ and $\ker \overline{p(\varphi)}^*(x)$ are totally isotropic 
and $\varphi$ is conjugate to a matrix 
\[
P =   \left (\begin{array} {cc} A & 0\\ 0 &  \overline{A^+} \end{array} \right ) \in \UG(2n, K).
\]
 Let $F = \{\lambda \in K; \overline{\lambda} = \lambda\}$.

\begin{lemma}   \label{Schur-17}
Let  $P \in \SUG(2n,K) =: \UG(2n, K) \cap \SL(2n, K)$. If $P$  is nonscalar,  
then $P$ is $\SUG$-conjugate to a matrix  $Q \in \SUG(2n,K)$ such that 
$\PC_n(Q)$ is nonscalar and has a prescribed determinant $a \in F$.
\end{lemma}

\begin{proof}
See Bünger \cite[Korollar 4.1.13]{Bunger-1997}.
\end{proof}

\begin{lemma}                            
	Let $h$ be a hyperbolic hermitian form on a finite-dimensional vector space $V$ over a field $K$. Let $\dim V \ge 4$. If $|K| \le 9$ let $\dim V \ge 8$.
	There exists a cyclic strictly hyperbolic conjugacy class 
	$\Omega$ of  $\SUG(V,K, h)$ with $\Omega = \Omega^{-1}$ such that $\Omega^2$
	contains all nonscalar transformations of $\SUG(V,K, h)$.
\end{lemma}

\begin{proof}
	Put $n = \frac{\dim V}{2}$.	Let $\Omega$ be the strictly hyperbolic  conjugacy class of $\UG(V,K, h)$ with minimal polynomial $p \overline{p}^*$, where $p$ is as follows:
\begin{enumerate}
		\item Let  $|K| \ge 16$. Let $\alpha \in F - \{0, \pm 1\}$. Let $\lambda \in K - F$ with $\lambda \overline{\lambda} \ne 1$. 
		Put $p(x) = (x-\lambda)(x-\overline{\lambda}) (x-\alpha)^{n-2}$.
		\item Let $K = \GF(9)$. Let $\lambda \in K$ with $\lambda \overline{\lambda} = -1$.
		The polynomial $x^3-x -1$ is irreducible in $\GF(3)[x]$, hence it is irreducible in $\GF(9)[x]$ as its degree is  odd.
			Let $x^2 + \mu x + 1$ be irreducible. Then $\mu \ne \overline{\mu}$.
    \begin{enumerate}
			\item If $n=2t+2 \ge 4$ let $p(x) = (x-\lambda)(x-\overline{\lambda}) (x^2 + \mu x + 1)^t$.
			\item If $n=2t+5 \ge 5$ let $p(x) = (x-\lambda)(x-\overline{\lambda}) (x^3-x -1) (x^2 + \mu x + 1)^t$.
   \end{enumerate}
		\item Let $K = \GF(4)$.  Let $\lambda \in K, \lambda \ne 0, 1$.
		The  polynomials $x^3+x+1, x^2 +x + \lambda, x^2 + x + \lambda^2$ and $x^2 + \lambda x + 1$ are irreducible, and $x^4 + x + 1 = (x^2 +x + \lambda) (x^2 + x + \lambda^2)$.   
    \begin{enumerate}
    	    \item If $n=3t + 2 \ge 5$ let $p(x) =  (x^3+x+1)^t (x^2 + \lambda x + 1)$.
    	    \item If $n=3t + 4 \ge 4$ let $p(x) = (x^3+x+1)^t (x^4 + x + 1)$.
    	    \item If $n=3t + 6 \ge 6$ let $p(x) = (x^3+x+1)^t (x^4 + x + 1)(x^2 + \lambda x + 1)$.
	\end{enumerate}
\end{enumerate}	
Obviously, $\Omega = \Omega^{-1}$. Further, $\Omega$ is a conjugacy class of $\SUG(2n,K)$:
	Let $\varphi \in \Omega$.  

First let  $|K| \ge 9$.
 Let $W = \ker (\varphi - \lambda) \oplus \ker (\varphi - \overline{\lambda}^{-1})$. Then $W$ is nondegenerate.
	It suffices to show that the centralizer of $\varphi|_W$ contains an element with arbitrary nonzero determinant $\delta$, where $\delta \overline{\delta} = 1$.

By Hilbert90, $\delta = \epsilon \overline{\epsilon}^{-1}$ for some $\epsilon \in K^*$. 
Obviously, the centralizer of $\varphi|_W$  contains a transformation with minimal polynomial $(x-\epsilon)  (x-\overline{\epsilon }^{-1})$.

So let $K = \GF(4)$. Let  $W = \ker (\varphi^2 + \lambda \varphi +1) (\varphi^2 + \overline{\lambda} \varphi +1)$ or $W = \ker (\varphi^4 + \varphi +1) (\varphi^4 +  \varphi^3 +1)$.
Then $W$ is nondegenerate, and the nontrivial  scalar transformations on $W$ are unitary. Further, $\det \alpha \Idm_4 = \alpha$ and $\det \alpha \Idm_8 = \alpha^2$.
	
Now let $\psi \in \SUG(V,K, h)$ be nonscalar. 
 By lemma \ref{Schur-17}, theorem \ref{BN}, and theorem \ref{main-1}, $\psi$ is $\SUG$-conjugate to a matrix $Q$ such that $\PC_n(Q)  =  RS$, where $R$ and $S$ are cyclic with $\mu(R), \mu(S) = p$.
Then
\[
	Q =   \left (\begin{array} {cc} A & B\\ C &  D \end{array} \right ) =
	\left (\begin{array} {cc} R & 0\\ CS^{-1} &  \overline{R}^+ \end{array} \right )
	\left (\begin{array} {cc} S & R^{-1}B\\ 0 &  \overline{S}^+ \end{array} \right ).
\]
Hence $\psi \in \Omega^2$.
\end{proof}


\ifdraft{\listoflabels}{}
\end{document}

\typeout{get arXiv to do 4 passes: Label(s) may have changed. Rerun}